\documentclass[12pt]{amsart}
\usepackage{amssymb} 
\newcommand{\diff}{\operatorname{Diff}}
\usepackage{hyperref}
\newcommand{\eps}{\varepsilon}
\renewcommand{\le}{\operatorname{\text{LE}}}
\newcommand{\Z}{\mathbb{Z}}
\newcommand{\N}{\mathbb{N}}
\newcommand{\R}{\mathbb{R}}
\newcommand{\pb}{\operatorname{Pb}}
\newcommand{\phc}{\operatorname{Phc}}
\newcommand{\nuh}{\operatorname{Nuh}}
\newcommand{\length}{\operatorname{length}}

\newcommand{\peri}{\operatorname{per}}

\newcommand{\zeroeq}{\stackrel{\rm{o}}{=}}
\newcommand{\transv}{\pitchfork}

\newtheorem{theorem}{Theorem}[section]
\newtheorem{claim}[theorem]{Claim}
\newtheorem{definition}[theorem]{Definition}
\newtheorem{proposition}[theorem]{Proposition}
\newtheorem{lemma}[theorem]{Lemma}

\newtheorem{corollary}[theorem]{Corollary}
\newtheorem{remark}[theorem]{Remark}

\title[A bound for internal radii]{A bound for internal radii of stable manifolds in terms of Lyapunov exponents}
\author{Jana Rodriguez Hertz}
\address{SUSTech, 1088 Xueyuan Road, Shenzhen, China.}
\email{rhertz@sustech.edu.cn}
\thanks{The author has been supported by the NSFC11871262, NSFC11871394, and NSFC12250710130 funds.}
\begin{document}
\begin{abstract}
We find some bounds for the internal radii of stable and unstable manifolds of points in terms of their Lyapunov exponents under the assumption of the existence of a dominated splitting. 
\end{abstract}
 \maketitle
 \section{Introduction}
 Let $M$ be a closed connected Riemannian manifold. Let $f\in\diff^{1+}(M)$. For any given invertible operator $A$, let us define its {\em conorm} $m(A)$ by $\|A^{-1}\|^{-1}$.
We say that $f$ admits a $\gamma$-dominated splitting with $\gamma>0$ if there is a $Df$-invariant splitting $TM= E^{-}\oplus E^{+}$ such that:
 
\begin{equation}\label{DS}
 \|Df_{-}(x)\|<e^{-2\gamma} m(Df_{+}(x)), 
\end{equation}
where $Df_\pm(x)=Df(x)|_{E^\pm}$.\par
The diffeomorphism $f$ has a dominated splitting if it admits a $\gamma$-dominated splitting for some $\gamma>0$.\par

We define the following exponents:
$$LE_{-}(x)=\limsup_{n\to \infty}\frac{1}{n}\sum_{k=0}^{n-1}\log \|Df_{-}(f^{k}(x))\|$$
and 
$$LE_{+}(x)=\liminf_{n\to\infty}\frac{1}{n}\sum_{k=0}^{n-1}\log m(Df_{+}(f^{k}(x)))$$

We define the Pesin stable and unstable manifolds by:
$$W^{-}(x)=\{y\in M:\limsup_{n\to \infty}\frac{1}{n}\log d(f^{n}(x), f^{n}(y))<0\}$$
and
$$W^{+}(x)=\{y\in M:\limsup_{n\to \infty}\frac{1}{n}\log d(f^{-n}(x), f^{-n}(y))<0\}$$
For any $x\in M$, we define the {\em maximal internal radius} of $W^\pm(x)$ by
$$R_\pm(x)=\inf\{\length(\alpha):\alpha(0)=x,\,\alpha(1)\in\partial W^\pm(x),\,\alpha(t)\in W^\pm(x)\,\forall t\in[0,1)\}.$$
A function $\phi:M\to\R$ is $(C,\alpha)$-H\"older if
$$|\phi(x)-\phi(y)|\leq Cd(x,y)^\alpha.$$

We find the following bound for the internal radius:
\begin{theorem}\label{internal.radius.x}
    Let $f\in\diff^{1+}$ be a diffeomorphism admitting a dominated splitting. Let $\mu$ be an invariant measure. Then, $\mu$-almost every $x$,
    \begin{equation}\label{inequality.LE}
\liminf_{N\to\infty}\frac{1}{N}\sum_{k=0}^{N-1}R_\pm(f^{k}(x))\geq\left(\frac{|\le_\pm(x)|}{C}\right)^\frac{1}{\alpha} 
\end{equation}
where $\|Df_-(x)\|$ and $m(Df_+(x))$ are $(C,\alpha)$-H\"older.
\end{theorem}
The variation of $x\mapsto E^\pm_x$ is H\"older, and this cannot be improved by increasing the differentiability of $f$.

 Let $p$ be a hyperbolic periodic point. The {\em ergodic homoclinic class} $\phc(p)$ of $p$ is defined as the intersection of the following two invariant sets:
$$\phc^-(p)=\{x:W^-(x)\transv W^+(o(p))\ne\emptyset\},\quad\text{and}$$
$$\phc^+(p)=\{x:W^+(x)\transv W^-(o(p))\ne\emptyset\},$$
where $o(p)$ denotes the orbit of $p$.\par
These sets were introduced in \cite{HHTU11}.

\section{Some notation and background}

\begin{theorem}[Pesin Stable Manifold Theorem \cite{pesin76}]\label{pesinstablemanifold} Let $f\in\diff^{1+}(M)$. Let $\mu$ be an invariant measure such that $\le_+(\mu)>0$. Then for each  sufficiently small $r>0$, there exists a measurable set $A_r$ with $\mu(A_r)>0$ such that:
$$R_+(x)\geq r\qquad\forall x\in A_r.$$
An analogous statement holds for $R_-(x)$ if $\le_-(\mu)<0$
\end{theorem}

We will state this classic lemma for later use:
\begin{lemma}[Kac's lemma]\label{kaclemma} Let $f\in\diff(M)$, $\mu$ an ergodic invariant probability measure, $\psi\in L^{1}(\mu)$, $A$ a measurable set with $\mu(A)>0$.
Define 
\begin{equation}\label{kacreturn}
\phi_{A}(x)=\min\{n>1:f^{n}(x)\in A\}.  
\end{equation}

Then
 $$\int\psi d\mu=\int_{A}\sum_{k=0}^{\phi_{A}(x)-1}\psi(f^{k}(x))d\mu.$$
\end{lemma}
\begin{proof}
 An interesting reference for this lemma is the unpublished notes \cite[Theorem 1.7]{SarigErgodicNotes}.
\end{proof}
The following criterion was introduced in \cite{HHTU11}:
\begin{theorem}[Criterion for ergodicity \cite{HHTU11}]\label{criterion.hhtu}
Let $f\in\diff^{1+}$ be a volume-preserving diffeomorphism and $p$ be a hyperbolic periodic point for $f$. If $m(\phc^-(p))>0$ and $m(\phc^+(p))>0$, then 
$$\phc^+(p)\zeroeq \phc^-(p)\zeroeq \phc(p)$$
is a hyperbolic ergodic component of $m$.
\end{theorem}
As a consequence, the Katok's closing lemma \cite{Katok1980} allows us to write the results in the well-known Pesin work \cite{pesin77} as:
\begin{theorem}[Pesin's spectral decomposition theorem]\label{pesin.spectral} Let $f\in\diff^{1+}(M)$ be a volume-preserving diffeomorphism. Let $\nuh(f)$ be the set of points without zero Lyapunov exponents. Then there exists a sequence of hyperbolic periodic points $p_n$ such that 
$$\nuh(f)\zeroeq \bigcup_{n\in\N}\phc(p_n),$$
where $\phc(p_n)$ are hyperbolic ergodic components of the volume measure.\par
If we call $$\Gamma^\pm(p)=\{x: W^\pm(x)\transv W^\mp(p)\ne\emptyset\} $$
and $\Gamma(p)=\Gamma^+(p)\cap \Gamma^-(p)$, then $f(\Gamma(f^k(p_n)))=\Gamma(f^{k+1}(p_n)$), and
$$f^{\peri(p)}|_{\Gamma(p)}\quad\text{is Bernoulli}.$$
\end{theorem}
\section{Average internal radius in terms of the Lyapunov exponents}

For each $r>0$, define the sets
$$G^\pm(r)=\{x:R_\pm(x)\geq r\}.$$
Theorem \ref{internal.radius.x} follows immediately from the following theorem.
\begin{theorem}\label{boundedradii}
Let $f\in\diff^{1+}(M)$ have a dominated splitting and let $\mu$ be an ergodic invariant probability measure such that $\le_{+}(\mu)>0$.
Then $\mu$-almost every $x\in M$, for {\em any} choice of initial internal radius $r_{0}(x)>0$ such that $x\in G^{+}(r_{0}(x))$, there is a sequence $r_{k}(x)$ (defined in \eqref{inductiveradii}) that satisfies:
\begin{enumerate}
 \item $f^{k}(x)\in G^{+}(r_{k}(x))$ for all $k\in \N$,
 \item $W^{+}_{r_{k}(x)}(f^{k}(x))\subset f(W^{+}_{r_{k-1}(x)}(f^{k-1}(x)))$ for all $k\in\N$,
 \item $$\le_{+}(\mu)\leq C\liminf_{N\to\infty}\frac{1}{N}\sum_{k=0}^{N-1}r_{k}(x)^{\alpha},$$
\end{enumerate}
where $C,\alpha>0$ are constants such that $\log\psi_{+}$ is $(C,\alpha)$-H\"older.
\end{theorem}
We denote by $B_\eps(x)$ the closed Riemannian ball of radius $\eps>0$ centered at $x$.

\begin{definition}\label{psi}
 For all $x\in M$, $N\in\N_{0}$, and $ \eps>0$, define:  
\begin{enumerate}
 \item  $\psi^N_{+}(x)=m(Df^N_{+}(x))$, 
 \item $\psi^N_{-}(x)=\|Df^N_{-}(x)\|$.
 \item $\psi^N_{+}(x,\eps)=\min\{\psi^N_{+}(y): y\in B_{\eps}(x)\}$, \\ 
 $\psi^N_{-}(x,\eps)=\max\{\psi^N_{-}(y): y\in B_{\eps}(x)\}.$
\end{enumerate}
When $N=1$, we omit 1 from the notation. 
\end{definition}

Before getting into the proof, let us see the following lemma:
\begin{lemma}\label{lemaures}
 Let $f\in\diff^{1}(M)$ with a dominated splitting. Let $r>0$ be such that $x\in G^{+}(r)$. If $m_{0}=\psi_{+}(x,r)$, then
 $f(x)\in G^+(m_{0}r)$. Moreover, $W_{m_{0}r}^{+}(f(x))\subset f(W^{+}_{r}(x))$. 
\end{lemma}
\begin{proof}
It is enough to see that $W^{+}_{m_{0}r}(f(x))\subset f(W^{+}_{r}(x))$. Consider a smooth path $\alpha:[0,1]\to W^{+}_{r}(x)$ such that $\alpha (0)=x$ and $\alpha (1)\in\partial W^{+}_{r}(x)$. Then
$$\length(f\circ\alpha)=\int_{0}^{1}\|Df(\alpha(t))\alpha'(t)\|dt\geq m_{0}\length(\alpha)\geq m_{0}r.$$
\end{proof}
\begin{proof}[ Proof of Theorem \ref{boundedradii}] 
Let $x\in M$ be any point such that $R_{+}(x)>0$. Choose any $r_{0}>0$ such that $r_{0}\leq R_{+}(x)$, and for each $k\in \N$ define inductively:
\begin{equation}\label{inductiveradii}
 r_{k}=\psi_{+}(f^{k-1}(x),r_{k-1})r_{k-1}=m_{k-1}r_{k-1}
\end{equation}
 
 By Lemma \ref{lemaures} $f^{k}(x)\in G^{+}(r_{k})$ for all $k\in \N_{0}$. Since $\log\psi_{+}$ is $(C,\alpha)$-H\"older, we have for all $N\in\N$:
\begin{equation}\label{THEinequality}
\frac{1}{N}\log \frac{r_{N}}{r_{0}}=\frac{1}{N}\sum_{k=0}^{N-1}\log\psi_{+}(f^{k}(x),r_{k})\geq \frac{1}{N}\sum_{k=0}^{N-1}\log\psi_{+}(f^{k}(x))-\frac{C}{N}\sum_{k=0}^{N-1}r_{k}^{\alpha}.
\end{equation}
\begin{claim} \label{claim}For each $\delta>0$ and $\mu$-almost every $x\in M$ there exists $N(x,\delta)\in\N$ such that for all $n\geq N(x,\delta)$
$$\le_{+}(\mu)\leq \delta +\frac{C}{n}\sum_{k=0}^{n-1}r_{k}^{\alpha}$$ 
\end{claim}
\begin{proof}[Proof of Claim \ref{claim}]Let $L=\max\{\log\psi_{+}(x):x\in M\}>0$. Consider $N(x,\delta)>0$ such that for all $n\geq N(x,\delta)$
$$\frac{1}{n}\sum_{k=0}^{n-1}\log\psi_{+}(f^{k}(x))>\le_{+}(\mu)-\frac{\delta}{2},\quad\text{and}$$
$$\le_{+}(\mu)\leq \frac{C}{n}\left(r_{0}e^{\frac{\delta}{2} n-L}\right)^{\alpha}.$$
If for some $n\geq N(x,\delta)$ we had 
\begin{equation}\label{novaaser}
\le_{+}(\mu)>\delta+\frac{C}{n}\sum_{k=0}^{n-1}r_{k}^{\alpha},  
\end{equation}

then by our choice of $N(x,\delta)$, we would have 
$$\frac{1}{n}\sum_{k=0}^{n-1}\log\psi_{+}(f^{k}(x))-\frac{C}{n}\sum_{k=0}^{n-1}r_{k}^{\alpha}>\frac{\delta}{2}.$$
Inequality \eqref{THEinequality} then would imply
$$r_{n}\geq e^{\frac{\delta}{2}n}r_{0}.$$
Now, from Formula \eqref{inductiveradii} we would have
$$r_{n-k}=\frac{r_{n}}{m_{n-1}\cdots m_{n-k}}\geq r_{0} e^{\frac{\delta}{2}n-kL}.$$
Then 
$$\sum_{k=0}^{n-1}r_k^\alpha\geq (r_0e^{\frac\delta2n})^\alpha\sum_{k=1}^ne^{-\alpha kL}.$$

Our assumption \eqref{novaaser} then would yield
$$\le_{+}(\mu)>\delta +\frac{C}{n}\left(r_{0}e^{\frac{\delta}{2}n-L}\right)^{\alpha},$$
contradicting our choice of $N(x,\delta)$. This proves the claim.
\end{proof}
The claim implies item (3) and Theorem \ref{boundedradii}.\end{proof}
\begin{remark}
 If $R_{+}(x)=\infty$ for a measurable positive measure set $A\subset M$, it follows from Lemma \ref{lemaures} that $R_{+}(x)=\infty$ on $f(A)$. Since $\mu$ is ergodic, $R_{+}(x)=\infty$ for $\mu$-almost every $x$.    
\end{remark}
 
\begin{theorem}\label{meaasurableradius}
 Let $f\in\diff^{1+}(M)$ be a diffeomorphism admitting a dominated splitting. Let $\mu$ be an ergodic invariant probability measure such that $\le_{+}(\mu)>0$. Then, there exists an $L^{1}(\mu)$ function $r_{+}:M\to(0,\infty)$ such that $x\in G^{+}(r_{+}(x))$ for $\mu$-almost every $x$, and
\begin{equation}\label{integrableradiuseq}
 \le_{+}(\mu)\leq \int\log\frac{r_{+}(x)}{r_{0}}d\mu+ C\int r_{+}(x)^{\alpha}d\mu
\end{equation}
where $C,\alpha>0$ are constants such that $\log\psi_{+}$ is $(C,\alpha)$
-H\"older. \end{theorem}
\begin{proof} Let $A_{r_0}$ be such that $\mu(A_{r_0})>0$, where $A_{r_0}$ is as in Theorem \ref{pesinstablemanifold}. For almost every $x\in M$, 
call $\phi:=\phi_{A_{r_0}}$ the measurable return function defined in \eqref{kacreturn} for the set $A_{r_0}$ and for $f^{-1}$ (do not confuse with $n(x)$). That is, 
$$\phi(x)=\min\{n>1: f^{-n}(x)\in A_{r_0}\}.$$
For all $x\in A_{r_0}$, define $r_{+}(x):=r_0>0$, an internal radius of $W^+(x)$. For all $x\in \phi^{-1}(\N)$, define $r_{+}(x):=r_{\phi(x)}$, where $r_{\phi(x)}$ is the one obtained in the recursive formula \eqref{inductiveradii}, that is:
$$r_{+}(x)=\prod_{k=0}^{\phi(x)-1}\psi_{+}(f^{-k}(x),r_{+}(f^{-k}(x)))r_0.$$
 It is easy to check that $r_{+}(x)$ is a measurable function and $x\in G^{+}(r_{+}(x))$ for $\mu$-almost every $x$. If $r_{+}$ is not in $L^{1}(\mu)$, one can easily take a truncation of $r_{+}$ that is in $L^{1}(\mu)$ and satisfies \eqref{integrableradiuseq} and $x\in G^{+}(r_{+}(x))$. So, assume $r_{+}$ is in $L^{1}(\mu)$. Hence, by Jensen's inequality, $r_{+}$ is in $L^{\alpha}(\mu)$. \par
 Now, by Kac's Lemma (Lemma \ref{kaclemma}), we have:
\begin{eqnarray*}
 \noalign{$\displaystyle{\le_{+}(\mu)-C\int r_{+}(x)^{\alpha}d\mu=}$}
 &=& \int_{\pb}\sum_{k=0}^{\phi(x)-1}[\log\psi_{+}(f^{-k}(x))-Cr_{+}(f^{-k}(x))^{\alpha}]d\mu\\
 &\leq& \int_{\pb}\sum_{k=0}^{\phi(x)-1}\log\psi_{+}(f^{-k}(x),r_{+}(f^{-k}(x))d\mu\\
 &=&\int_{\phi^{-1}(\N)}\log\frac{r_{+}(x)}{r_{0}}d\mu = \int \log\frac{r_{+}(x)}{r_{0}}d\mu
 \end{eqnarray*}
\end{proof}
\begin{remark} As a corollary of Jensen's inequality, under the assumptions of Theorem \ref{meaasurableradius}, we get
$$\le_{+}(\mu)\leq \log \int \frac{r_{+}(x)}{r_{0}}d\mu+C\left(\int {r_{+}(x)}d\mu\right)^{\alpha}.$$
Also, if we choose $$0<r_{0}\leq \left(\frac{\le_{+}(\mu)}{C}\right)^{\frac{1}{\alpha}},$$
then it follows that $\int r_{+}(x) d\mu\geq r_{0}$, otherwise, we would get a contradiction with the inequality above. 
\end{remark}
\begin{corollary}\label{cor.R+}
    Let $\mu$ be an ergodic measure such that $\le_+(\mu)>0$. Then $R_+$ is a measurable function. 
\end{corollary}
\begin{proof}
For $\mu$-almost every $x$ and each $k\in \Z$ there exists $r_{+,k}\in L^1$ such that $r_{+,k}(f^k(x))=R_+(f^k(x))$, and $y\in G^+(r_{+,k}(y))$ $\mu$-almost every $y$.
Take $r=\sup_{k\in\Z}r_{+,k}$. Then $r$ is a measurable function and $r(y)=R_+(y)$ $\mu$-almost every $y$.
\end{proof}
\section{Internal radii for periodic points}
The following corollary follows immediately from Theorem \ref{boundedradii}:
\begin{corollary} \label{inequalitymaximalradii} Under the hypothesis of Theorem \ref{boundedradii}, if $p$ is a periodic point such that $\le_{+}(p)>0$, then
$$\le_{+}(p)\leq \frac{C}{\peri(p)}\sum_{k=0}^{\peri(p)-1}R_{+}(f^{k}(p))^{\alpha}\leq C\left(\frac{1}{\peri(p)}\sum_{k=0}^{\peri(p)-1}R_{+}(f^{k}(p))\right)^{\alpha}.$$
\end{corollary}

\begin{proof}
 Let $r_{0}=R_{+}(p)$ and let us do the inductive procedure in \eqref{inductiveradii}. Then we obtain 
 $$\frac{1}{N}\log\frac{r_{N}}{R_{+}(p)}=\frac{1}{N}\sum_{k=0}^{N-1}\log m_{k}\leq 0,$$
 where $N=\peri(p)$.
 Otherwise we would obtain that $r_{N}>R_{+}(p)$, which is a contradiction. We also have $\psi_{+}(f^{k}(p),R_{+}(f^{k}(p)))\leq m_{k}$ for all $k\in[0,\peri(p)-1]$.\par
Due to the $(C,\alpha)$-H\"olderness of the function $\log\psi_{+}$, the following holds
  $$\log\psi_{+}(f^{k}(p),R_{+}(f^{k}(p)))\geq \log\psi_{+}(f^{k}(p))-CR_{+}(f^{k}(p))^{\alpha}.$$
  The result then follows. 
\end{proof}

\begin{proposition}\label{internalradiusbound} Let $p$ be a hyperbolic periodic point of period $N$, then
 $$R_{+}(p)\geq d(p,M\setminus A^{+}(N)),$$
 
 where $A^{+}(N)=\{x\in M: \log\psi^N_+(x)>0\}$.
\end{proposition}
\begin{proof} 
It follows from Lemma \ref{lemaures} for $f^{N}$ that $\psi^N_+(p,R_+(p))\leq 1$, for otherwise we would obtain a contradiction with our choice of $R_{+}(p)$. This implies that 
$W^+_{R_{+}(p)}(p)\cap (M\setminus A^{+}(N))\ne\emptyset$. This implies that $d(p,M\setminus A^{+}(N))\leq R_{+}(p)$.\par
\end{proof}
\begin{remark}\label{remark.pper}
\begin{enumerate}
    \item Maybe it is handy to note the following: if $\overline{W^{+}_{\eps}(p)}\subset A^{+}(\peri(p))$, then $R_{+}(p)>\eps$.
\item 
\end{enumerate}
\end{remark}

\section{Time bounds}
\begin{definition}[Pesin blocks] Given $f\in\diff^{1}(M)$ admitting a $\gamma$-dominated splitting, the {\em Pesin blocks} are the sets of the form:
$$\pb^+_N(\gamma)=\left\{x\in M^+: \frac{1}{n}\sum_{k=0}^{n-1}\log \psi_+f^k(x))\geq\gamma/2\quad\forall n\geq N\right\},$$
$$\pb^-_N(\gamma)=\left\{x\in M^-: \frac{1}{n}\sum_{k=0}^{n-1}\log \psi_-(f^k(x))\leq-\gamma/2\quad\forall n\geq N\right\}.$$
\end{definition}
Pesin blocks are closed sets where there is ``uniform hyperbolicity'', but at the cost of not being invariant. 

\begin{proposition}
    For all $x\in\pb^+_N(\gamma)\cap G^+(R_0e^{-K\gamma/4})$, 
    $$\sup_{0\leq k\leq \max(K,N)}R_+(f^k(x))\geq R_0,$$
where $R_0=(\frac{\gamma}{4C})^\frac{1}{\alpha}$.
\end{proposition}
\begin{proof}
    Suppose 
    $$\sup_{0\leq k\leq \max(K,N)}R_+(f^k(x))< R_0$$
and call $T=\max(K,N)$. Then 
$$\frac1T\log\frac{R_+(f^T(x))}{R^+(x)}\geq \frac{\gamma}2-\frac{C}{T}\sum_{k=0}^{T-1}\sup R_+(f^k(x))^\alpha\geq \frac{\gamma}{4}.$$
This implies that
$$R_+(f^T(x))\geq R^+(x)\exp\left(\frac{\gamma}{4}K\right)\geq R_0.$$
This produces a contradiction. 
\end{proof}


\bibliographystyle{alpha}
\bibliography{2023jana}

 \end{document}